\documentclass{article}
\usepackage{amssymb}
\usepackage{latexsym}
\usepackage{graphicx}

\newtheorem{theorem}{Theorem}

\newtheorem{conjecture}[theorem]{Conjecture}
\newtheorem{corollary}[theorem]{Corollary}

\newtheorem{definition}[theorem]{Definition}

\newtheorem{lemma}[theorem]{Lemma}

\newtheorem{proposition}[theorem]{Proposition}

\newenvironment{proof}[1][Proof]{\noindent\textbf{#1.} }{\ \rule{0.5em}{0.5em}}
\def\lab(#1)#2{\put(#1){\makebox(0,0)[c]{#2}}}
\begin{document}

\title{Approximating Minimum Steiner Point Trees in Minkowski Planes\thanks{This research was supported by an ARC
Discovery Grant.}}

\author{M.~Brazil \and C.~J.~Ras \and  D.~A.~Thomas}

\date{}
\maketitle

\begin{abstract}
Given a set of points, we define a minimum Steiner point tree to be a tree interconnecting these points and possibly some additional points such
that the length of every edge is at most 1 and the number of additional points is minimized. We propose using Steiner minimal trees to
approximate minimum Steiner point trees. It is shown that in arbitrary metric spaces this gives a performance difference of at most $2n-4$,
where $n$ is the number of terminals. We show that this difference is best possible in the Euclidean plane, but not in Minkowski planes with
parallelogram unit balls. We also introduce a new canonical form for minimum Steiner point trees in the Euclidean plane; this demonstrates that
minimum Steiner point trees are shortest total length trees with a certain discrete-edge-length condition.\\

\noindent \em Keywords: minimum Steiner point trees, bounded edge-length, Minkowski planes
\end{abstract}

\section{Introduction}
Given a metric space $(S,d)$ with metric $d$, a set of points $N \subseteq S$, and an $R\in \mathbb{R}$, the \textit{minimum Steiner point tree
problem} (MSPT problem) asks for a set $U\subset S$ and a tree connecting $N\cup U$ such that no edge is longer than $R$ and $|U|$ is minimized.
Clearly we may assume that $R=1$. An optimal solution is called a $d$-MSPT, or just an MSPT if the context is clear. MSPTs have applications in
the deployment and augmentation of wireless sensor networks, VLSI design, and wavelength-division multiplexing networks - see for instance
\cite{bib10,bib8,bib7,bib9}.

The MSPT problem was first described by Sarrafzadeh and Wong in \cite{bib3}, where they showed that it is NP-complete in both the $\ell_1$ and
$\ell_2$ metric. Consequently a fair amount of research has been directed towards finding good heuristics. In \cite{bib4} the minimum spanning
tree (MST) heuristic was introduced (note that there they refer to the MSPT problem as the \textit{Steiner tree problem with minimum number of
Steiner points and bounded edge length}, or STP-MSPBEL). This heuristic simply subdivides all edges of an MST that are longer than one unit,
resulting in an approximate MSPT solution within polynomial time. Mandoiu and Zelikovsky \cite {bib6} prove that, in any metric space, the
performance ratio of the MST heuristic is always one less than the maximum possible degree of a minimum-degree MST spanning points from the
space. This gives an approximation ratio of four in the Euclidean plane and three in the rectilinear plane. Chen et al. \cite{bib5} provide an
improved approximation scheme, partly based on the MST heuristic, which has a performance ratio of three in the Euclidean plane.

The MSPT problem may be seen as a variant of the classical Steiner tree problem, which asks for a shortest tree interconnecting $N\subseteq S$
where any number of additional points may be introduced. An optimal solution to this problem is called a Steiner minimal tree ($d$-SMT or just
SMT). As $R$ tends to zero an SMT with subdivided edges becomes an optimal solution to the MSPT problem. This leads us to the question: would
the SMT approximation for the MSPT problem be a practical and accurate heuristic? Certainly we do not have effective algorithms for calculating
SMT's in every metric, in fact the problem is NP-hard. However, in the Euclidean plane and other fixed orientation metrics, Warme, Winter, and
Zachariasen \cite{bib1,bib2} have developed practical, fast and optimal SMT algorithms, namely the GeoSteiner algorithms. These algorithms can
comfortably solve most instances of up to a few thousand uniformly distributed terminals. However, as should be expected, it is possible to
construct terminal-sets that take much longer to process; for instance, GeoSteiner cannot efficiently find an SMT when just one hundred
terminals are located at the vertices of a regular square lattice in the Euclidean plane (although these instances can be solved in polynomial
time by the algorithms of Brazil et al. \cite{Brazil2}).

In this paper we define and analyze the \textit{SMT heuristic} for MSPTs. We provide a small upper bound (in terms of $|N|$) for the performance
difference of the SMT heuristic in any normed plane, and show that this bound is best possible in the Euclidean plane. We then show that, in the
special case $|N|=3$, the upper bound is tight in a Minkowski plane with unit ball $B$ if and only if $B$ is not a parallelogram. For the
Euclidean and rectilinear planes a brief comparison between the SMT heuristic and current best possible heuristics is given. Then we prove that
the performance ratio of the SMT heuristic improves as $R$ decreases (or equivalently, as the terminals become further apart). This paper also
explores the possibility of restating the Euclidean MSPT problem in terms of shortest total length, leading to a new MSPT canonical form.
Finally, we state a number of strong conjectures on the relationship between the Steiner tree problem and the MSPT problem.

\section{Preliminaries}
Let $(S,d)$ be a metric space with metric $d$, and consider a set $N \subseteq S$. The \textit{Steiner tree problem} asks for a shortest tree
interconnecting $N$, where extra nodes $W \subset S$ are introduced if they reduce the total length. Introducing degree-one or degree-two nodes
will not reduce total length, henceforth for the Steiner tree problem we assume all added nodes are of degree at least three. The nodes in $N$
are called \textit{terminal points} and the nodes in $W$ are called \textit{Steiner points}.

In general metric spaces there may be instances of the MSPT problem that have no solution; consider, for instance, the case when $N=S$ and
$\min\{d(x,y):x,y\in S\}>1$. Henceforth we will assume the following: $S=\mathbb{R}^2$, $\vert N\vert$ is finite, and $d$ is a norm. In other
words, we will only be considering the finite MSPT problem in \textit{Minkowski planes}. In our discussions we distinguish between the concept
of a \textit{free node} and an \textit{embedded node}. In other words any tree may be considered as a topological graph structure only, or as an
embedded network. Embedded nodes are denoted by bold letters (as is common when representing vectors). An embedded set of terminals
\textit{admits} a tree with property $P$ if there exists a tree $T$ interconnecting the terminals such that $T$ has property $P$.

Two standard techniques for shortening an embedded tree are \textit{splitting} and Steiner point \textit{displacements}. To \textit{split} a
node $v$ one disconnects two or more of the edges at $v$ and connects them instead to a new Steiner point, connected to $v$ by an extra edge. To
\textit{displace} a Steiner point one simply embeds it at any new point in the plane without changing the topology of the tree. If no shortening
of a tree is possible when splitting or Steiner point displacements are allowed, then the tree is called a \textit{Steiner tree}. Note that an
SMT is always a Steiner tree. A \textit{full Steiner tree} is a Steiner tree where every terminal is of degree one and every Steiner point is of
degree three. A full Steiner tree has exactly $\vert N\vert -2$ Steiner points and $2\vert N \vert-3$ edges. A \textit{cherry} of a full Steiner
tree is the subtree induced by two terminals and their mutually adjacent Steiner point. Every full Steiner tree has at least two cherries. We
refer the reader to \cite{bib13} and \cite{bib14} for more background on Steiner trees.

Given two points $x,y \in S$, we denote the edge $e$ between them by $e=xy$, and we use the standard notation $\vert e \vert$ to denote
$d(x,y)$. Any Steiner tree can be viewed as a candidate MSPT if we simply subdivide, or \textit{bead}, edges that are longer than one unit.
Formally, \textit{beading} is the process whereby for every edge $e$, $\lceil \vert e \vert \rceil - 1$ equally spaced degree-two nodes lying on
$e$ are included (along with the elements of $W$) in the set $U$ of extra MSPT nodes. In general, any tree can be viewed as an MSPT candidate if
we partition its nodes into a set $N$ of terminals and a set $W$ of Steiner points of degree at least three, and then bead any edges that are
too long. Consequently, when constructing an MSPT on a given set $N$, we are mainly concerned with finding the elements of $W$, i.e., the
elements of $U$ that have degree at least three; clearly degree-one nodes will not occur in $U$ and degree-two nodes in $U$ only arise from
beading. Henceforth, degree-two nodes in $U$ will not be considered as part of the topology of the MSPT. All nodes in $U$ will be referred to as
\textit{beads} and, specifically, the nodes in $W$ will be called \textit{Steiner beads}. The procedure of constructing an SMT in order to
approximate an MSPT will be referred to as the \textit{SMT heuristic}.

Let $T$ be any tree with node-set partitioned into terminals $N$ and Steiner beads $W$. Let $n=\vert N \vert$. Then $T^*$ is the tree that
results by splitting nodes of $T$ until every terminal is of degree one and every Steiner bead is of degree three (i.e., $T^*$ is a full Steiner
tree). New nodes are not displaced from their original positions, in other words some zero edge-lengths may be introduced and the total length
of $T$ does not change. See Figure \ref{figureSat} as an example; here $t$ is a degree-four terminal, $s$ is a Steiner point, and after
splitting $t$ we have three zero-length edges (depicted by broken lines).

\begin{figure}[htb]
\begin{center}

\includegraphics[scale=0.4]{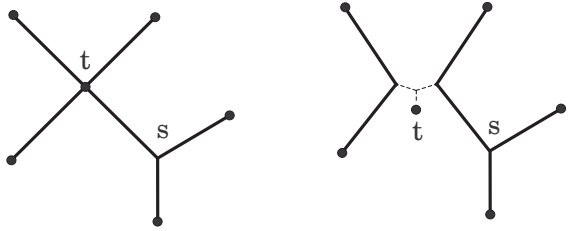}

\end{center}
\caption{Conversion to a full Steiner tree. \label{figureSat}}
\end{figure}

Let the edge-set of $T^*$ be $E(T^*)=\{e_1,...,e_m\}$, where $m=2n-3$. Then the \textit{bead count} of $T$ is $\mathrm{beads}(T)=\vert U \vert=
n-2+ \sum_{i=1}^m \left(\lceil \vert e_{i} \vert \rceil -1 \right)=1-n+\sum_{i=1}^m \lceil \vert e_{i} \vert \rceil$. In other words, by
considering $T^*$ rather than $T$ we get a formula for $\mathrm{beads}(T)$ that does not depend on the number of Steiner beads of $T$; this
formula works because every time a node is split (creating a new Steiner bead) we introduce a zero-length edge which in effect cancels the count
of this Steiner bead. We can now reformulate the MSPT problem as follows. Let $N$ be a subset of $S$. Find a $W\subset S$ and a tree $T$
interconnecting $N \cup W$ such that every node in $W$ is of degree at least three and $\mathrm{beads}(T)$ is a minimum over all trees
interconnecting $N$.

\section{The Upper Bound in any Normed Plane}\label{sec1}
In this section we provide an upper bound for the performance difference of the SMT heuristic in any normed plane. Let $N$ be a set of $n$
terminals in a normed plane $(\mathbb{R}^2,d)$. We use $T_{\mathrm{opt}}$ to denote an MSPT on $N$ and $T_{S}$ to denote an SMT on $N$. We need
the following lemma before we prove our main result:

\begin{lemma}\label{Lemma1} If $i,k$ are real numbers then $\lceil i+k \rceil - \lceil i \rceil = \lceil k \rceil$ or $\lceil k \rceil -1$
(equivalently $\lfloor k \rfloor$ or $\lfloor k \rfloor +1$), with $\lceil i+k \rceil -\lceil i \rceil = k$ if $k$ is an integer.
\end{lemma}
\begin{proof}
Suppose that $\lceil i \rceil = i +\varepsilon_i$ and $\lceil k \rceil = k + \varepsilon_k$ where $0 \leq \varepsilon_i,\varepsilon_k < 1$. Then
$\lceil i+k \rceil = \lceil i \rceil + \lceil k \rceil - \lfloor \varepsilon_i + \varepsilon_k \rfloor$, from which the result follows.
\end{proof}

Suppose that $E(T_{\mathrm{opt}}^*)=\{e_1,...,e_m\}$ and $E(T_{S}^*)=\{a_{1},...,a_m\}$. Since $T_S$ is a shortest total length tree
interconnecting $N$, we have:
\begin{eqnarray}
  \sum\limits_{i=1}^m \vert a_i \vert \leq \sum\limits_{i=1}^m \vert e_i \vert \label{eq0}
\end{eqnarray}

We can therefore partition the set $\{1,...,m\}$ as follows: let $\{1,...,m\}= I \cup D$ such that $\vert e_i \vert = \vert a_i \vert+p_i$ for
$i \in I$ and $\vert e_i \vert = \vert a_i \vert -p_i$ for $i\in D$. Here each $p_i$ is a non-negative real number and the cardinality of $D$,
but not $I$, may be zero.  We further partition $I$ into $I_Z$ and $I_{Z}^{\prime}$ (where $I_Z$ may be empty) such that $i \in I_Z$ if and only
if $\vert a_i \vert$ is an integer. We similarly partition $D$ into $D_Z$ and $D_{Z}^{\prime}$. From (\ref{eq0}) it follows that $\sum
\limits_{i\in I} p_i \geq \sum\limits_{i\in D} p_i$ - a result that is central to the proof of Proposition \ref{mainUpperProp}.

\begin{lemma}\label{Lemma2} If all edges of both $T_{S}$ and $T_{\mathrm{opt}}$ have integer length, then $T_{\mathrm{opt}}$
is also an SMT on $N$, and $\mathrm{beads}(T_S)=\mathrm{beads}(T_{\mathrm{opt}})$ .
\end{lemma}
\begin{proof}
Let $T$ be any Steiner tree on $N$ such that all edges of $T$ have integer length, and let $L(T)$ be the total length of $T$. Then
$\mathrm{beads}(T)= L(T) - n +1$. The lemma immediately follows since $L(T_{\mbox{\scriptsize opt}}) \geq L(T_S)$.
\end{proof}

\begin{proposition}\label{mainUpperProp}
$\mathrm{beads}(T_S)-\mathrm{beads}(T_{\mathrm{opt}}) \leq \max\{2n-4-j, 0\}$, where $j$ is the number of integer-length edges in $E(T_{S}^*)$.
\end{proposition}
\begin{proof}
\begin{eqnarray*}
\mathrm{beads}(T_S)-\mathrm{beads}(T_{\mbox{\scriptsize opt}}) &=&  \left[1-n+\sum\limits_{i=1}^m \lceil \vert a_{i} \vert
\rceil\right]-\left[1-n+\sum\limits_{i=1}^m
\lceil \vert e_{i} \vert\rceil\right]\\
&=& \sum\limits_{i=1}^m \lceil \vert a_{i} \vert \rceil-\sum\limits_{i=1}^m \lceil \vert e_{i} \vert \rceil\\
&=& \sum\limits_{i \in D}\{ \lceil \vert a_{i} \vert\rceil-\lceil\vert a_{i} \vert-p_i\rceil \}- \sum\limits_{i \in I}\{ \lceil \vert a_{i}
\vert+p_i\rceil-\lceil\vert a_{i} \vert\rceil \}.\\
\end{eqnarray*}

Using Lemma~\ref{Lemma1} we obtain:
\begin{eqnarray}
\mathrm{beads}(T_S)-\mathrm{beads}(T_{\mbox{\scriptsize opt}}) &\leq& \sum\limits_{i\in D_Z} \lfloor p_i \rfloor + \sum\limits_{i\in
D_{Z}^{\prime}} \left( \lfloor
p_i \rfloor+1\right) - \sum\limits_{i \in I_Z} \lceil p_i\rceil - \sum\limits_{i \in I_{Z}^{\prime}} \left(\lceil p_i\rceil -1\right) \nonumber \\
&=& \vert D_{Z}^{\prime} \vert + \vert I_{Z}^{\prime} \vert +\sum\limits_{i \in D} \lfloor p_i \rfloor - \sum\limits_{i\in I} \lceil p_i\rceil \nonumber \\
&\leq & m-j+\sum\limits_{i \in D} p_i - \sum\limits_{i \in I} \lceil p_i\rceil \label{eq1} \\
&\leq & m-j+\sum\limits_{i \in I} p_i - \sum\limits_{i \in I} \lceil p_i\rceil \nonumber \\
&\leq & m-j  \label{eq2} \\
&=& 2n-3-j. \nonumber
\end{eqnarray}

\noindent We now consider a number of cases showing that either (\ref{eq1}) or (\ref{eq2}) is a strict inequality or
$\mathrm{beads}(T_S)=\mathrm{beads}(T_{\mbox{\scriptsize opt}})$. Together, these imply the statement of the lemma.

\textsc{Case 1:} Suppose there exists an edge in $T_S$ that is longer than some edge in $T_{\mbox{\scriptsize opt}}$ and such that the
difference between the lengths of the two edges is not an integer. Then there exists an assignment of labels $\{e_i\}$ to the edges of
$T_{\mbox{\scriptsize opt}}$ and labels  $\{a_i\}$ to the edges of $T_S$ such that $p_i \not\in \mathbb{Z}$ for some $i \in D$. Hence
Inequality~(\ref{eq1}) is strict.

\textsc{Case 2:} If there exists an edge in $T_S$ that is shorter than some edge in $T_{\mbox{\scriptsize opt}}$ and such that the difference
between the lengths of the two edges is not an integer, then by the same argument as in Case~1, we can assume Inequality~(\ref{eq2}) is strict.

\textsc{Case 3:} The only remaining possibility is that the difference in length between each edge in $T_S$ and each edge in
$T_{\mbox{\scriptsize opt}}$ is an integer. This means there exists an $\varepsilon \in [ 0, 1)$ such that the length of every edge in both
trees is an integer plus $\varepsilon$. If $\varepsilon=0$ then $\mathrm{beads}(T_S)-\mathrm{beads}(T_{\mbox{\scriptsize opt}})=0$ by
Lemma~\ref{Lemma2}. If $\varepsilon \not= 0$ then we can move any Steiner point in $T_{\mbox{\scriptsize opt}}$ by a sufficiently small distance
($>0$) such that the length of at least one edge changes without changing the bead count of $T_{\mbox{\scriptsize opt}}$. Hence we can then
apply Case~1 or 2.

\end{proof}

\begin{corollary} $\mathrm{beads}(T_S) - \mathrm{beads}(T_{\mathrm{opt}}) \leq 2n-c-3$ where $c$ is the number of full components of $T_S$.
\end{corollary}
\begin{proof}
Note that every terminal $x$ of degree $\deg(x)$ is split $\deg(x)-1$ times to produce $T_{S}^*$, i.e., each terminal $x$ produces $\deg(x)-1$
zero-length edges after all splits. Clearly also $c=\sum\limits_{x \in N}\{\deg(x)-1\}+1$.
\end{proof}

\begin{corollary}\label{EqCorollary} If $T_S$ has at most one edge with non-integer length then $\mathrm{beads}(T_S)=\mathrm{beads}(T_{\mathrm{opt}})$.
\end{corollary}

Du et al. \cite{bib5,bib15} provide approximations for the MSPT problem that give performance ratios with upper bounds of three in the Euclidean
plane and two in the rectilinear plane. Their algorithms are based on the MST heuristic and therefore run in polynomial time. If we rewrite our
performance difference to get the bounded ratio $\frac{\mathrm{beads}(T_S)}{\mathrm{beads}(T_{\mathrm{opt}})} \leq
1+\frac{2n-4}{\mathrm{beads}(T_{\mathrm{opt}})}$ we see that the performance ratio of the SMT heuristic has a smaller upper bound than the
heuristics of Du et al. when $\mathrm{beads}(T_{\mathrm{opt}}) > n-2$ in the Euclidean plane, and $\mathrm{beads}(T_{\mathrm{opt}}) > 2n-4$ in
the rectilinear plane. Since $\mathrm{beads}(T_{\mathrm{opt}})$ increases as the minimum distance between any pair of terminals increases, we
arrive at the intuitive fact that the performance of the SMT heuristic improves as the terminal configuration becomes more sparse. If $R$ was
not fixed then we would arrive at the same result by decreasing $R$. During this limiting process the upper bound of the ratio
$\frac{\mathrm{beads}(T_S)}{\mathrm{beads}(T_M)}$, where $T_M$ is an MST, tends towards the well-known Steiner ratio. This gives a limiting
upper bound of $\frac{\mathrm{beads}(T_S)}{\mathrm{beads}(T_M)} \leq \frac{\sqrt{3}}{2}$ in the Euclidean plane, which serves as a comparison
between the performances of the SMT heuristic and the standard MST heuristic.

We mention once again that the SMT heuristic does not run in polynomial time. However, for $n$ up to a few thousand nodes (uniformly distributed
in a square) the GeoSteiner algorithms will produce solutions in reasonable running time for the Euclidean and rectilinear plane~\cite{bib1}.
This makes the SMT heuristic a tool worthy of consideration for applications where optimization is required during an initialization process
(such as deployment). In fact, one should consider this heuristic for any process where the cost benefit of a more accurate algorithm justifies
a possible time delay.

It should also be noted that SMTs can be approximated arbitrarily closely in polynomial time. The polynomial-time approximation scheme (PTAS)
developed by Arora~\cite{biba1, biba2} works for any norm, and allows one to construct a solution to the SMT problem that is within a factor of
$1 + \epsilon$ from optimality in polynomial time for any fixed $\epsilon > 0$. In theory this gives a good polynomial-time heuristic for the
MSPT problem, obtained by replacing the SMT by its $1 + \epsilon$ approximation. There is, however, a difficulty with this approach in that the
degree of the polynomial for small values of $\epsilon$ is too large to make the algorithm practical.

\section{The Euclidean Plane}\label{sec2}
The aim of this section is to show that the performance difference from Proposition \ref{mainUpperProp} is best possible in the Euclidean plane.
We begin with a few definitions and preliminary results. Due to minimality of total length, any two adjacent edges of a Euclidean Steiner tree
meet at an angle of at least $120^\circ$. This implies that the degree of any terminal is no more than $3$, and the degree of any Steiner point
is exactly $3$. Let $T$ be a full Euclidean Steiner tree on a set of embedded terminals. To \textit{sprout} new terminals from a given terminal
$\mathbf{t}$ of $T$ with incident edge $e$ one replaces $\mathbf{t}$ by a Steiner point $\mathbf{s}$ and embeds two new terminals
$\mathbf{t}_1,\mathbf{t}_2$ adjacent to $\mathbf{s}$ such that the two new edges $\mathbf{st}_1$ and $\mathbf{st}_2$ each form $120^\circ$
angles with $e$ and with each other - see Figure \ref{figureSplit}. We denote by $L(T)$ the total Euclidean edge length of $T$. If $N$ is a set
of embedded terminals then $T_S$ will denote a Euclidean SMT on $N$ and $T_{\mathrm{opt}}$ will denote a Euclidean MSPT on $N$. As usual we let
$n=\vert N \vert$. The next proposition shows that we can use sprouting to create full SMTs with any given topology. It is a fundamental result
and is almost certainly known, but does not appear to have been explicitly written up in the literature before now.

\begin{figure}[htb]
\begin{center}

\includegraphics[scale=0.4]{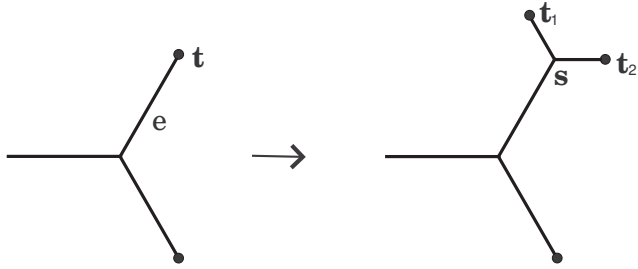}

\end{center}
\caption{Sprouting new terminals.  \label{figureSplit}}
\end{figure}

\begin{proposition} Given any full Steiner topology, there exists a set of embedded terminals $N$ such that the SMT for $N$ has the given topology
and is unique. Furthermore, such trees can be explicitly constructed for any given topology.
\end{proposition}

\begin{proof}
Let $G_n$ be a full Steiner topology on $n$ terminals. We will show how a suitable set of embedded terminals $N_n$ can be constructed by
induction on $n$, where the inductive step involves sprouting new terminals. Note that the construction is trivial if $n=1,2$ or $3$. The
inductive claim is as follows.

\textbf{Claim}: For any full Steiner topology $G_i$ on $i$ terminals (with $i\geq 4$), there exists a set of embedded terminals $N_i$ and a real
number $f_i >0$ such that \begin{enumerate}
    \item an SMT, $T_i$, for $N_i$ has topology $G_i$, and
    \item if $T_i'$ is a Steiner tree for $N_i$ such that the topology of $T_i'$ is not $G_i$, then $L(T_i') - L(T_i) \geq f_i$.
\end{enumerate}

For the base case of the claim ($i=4$), choose $N_4$ to be the four points with coordinates $(\pm 1, \pm \sqrt{3}/2)$. It is easily checked that
the SMT $T_4$ for $N_4$ has Steiner points $(\pm 1/2, 0)$ and length $5$ (see Figure \ref{figureBC}). The shortest Steiner tree $T'_4$ with a
different topology is full with Steiner points $(0, \pm (\sqrt{3}/2- 1/\sqrt(3)))$  and length $L(T'_4)= 3\sqrt{3}$. So we can choose $f_4=
3\sqrt{3}-5>0$. Up to relabelling of the terminals, there is only one full topology for $i=4$, so this completes the base case.

\begin{figure}[htb]
\begin{center}

\includegraphics[scale=0.4]{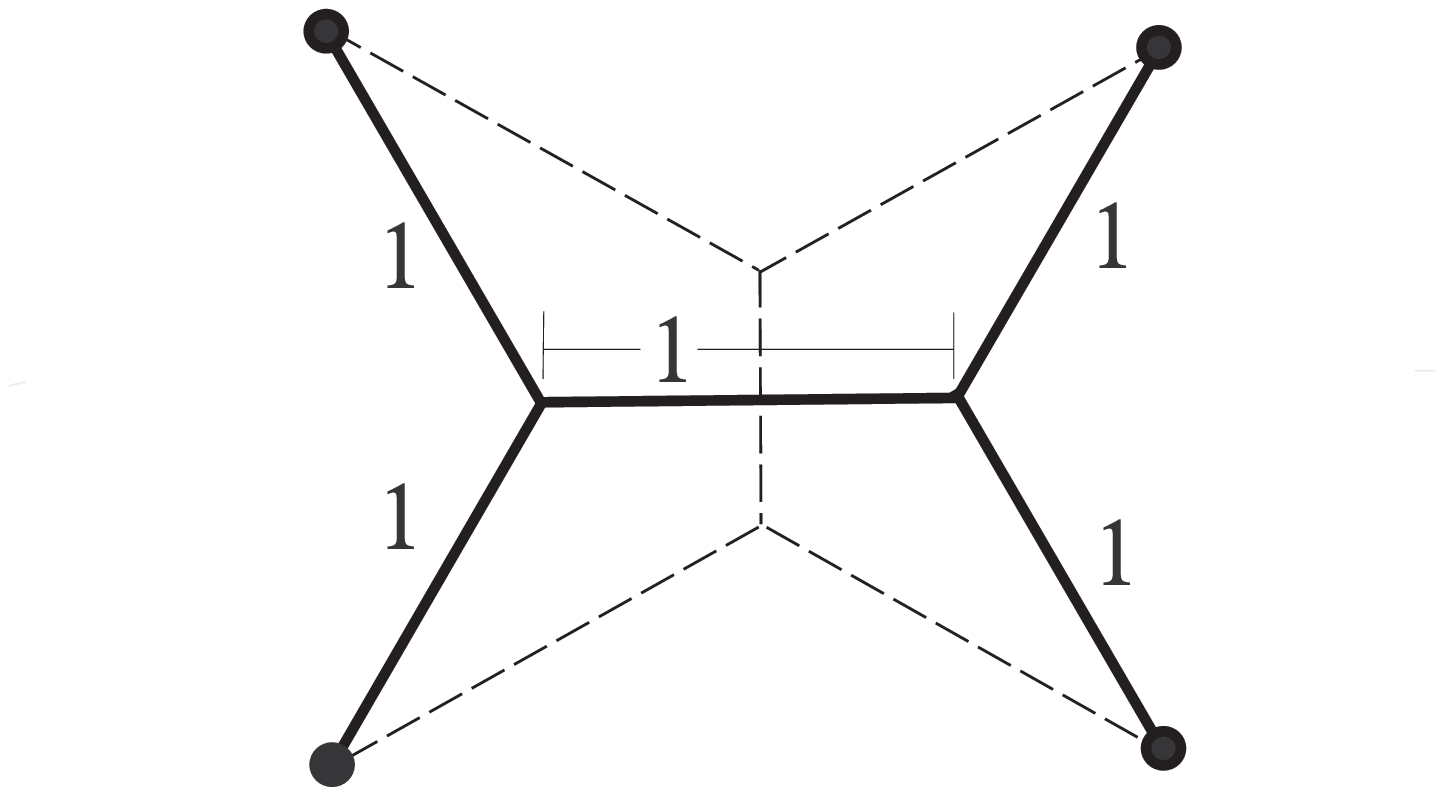}

\end{center}
\caption{Base case.  \label{figureBC}}
\end{figure}

We now establish the inductive step for ($i=n$), where we assume that the claim holds for $i=n-1$. Given a full Steiner topology $G_n$ ($n>4$),
this topology contains at least one cherry. Replacing such a cherry by a single terminal $t^*$ gives a full Steiner topology $G_{n-1}$ on $n-1$
terminals. By the inductive assumption there exists an embedded terminal set $N_{n-1}$ with unique SMT $T_{n-1}$ which has topology $G_{n-1}$
and a corresponding constant $f_{n-1} > 0$. Let $\mathbf{t}$ be the embedded terminal corresponding to $t^*$ and create a new Steiner tree as
follows.

We sprout new terminals $\mathbf{t}_n$ and $\mathbf{t}_{n-1}$ from $\mathbf{t}$, with $\mathbf{t}$ replaced by a Steiner point $\mathbf{s}$,
such that $\vert\mathbf{st}_{n-1}\vert=\vert\mathbf{st}_n\vert= f_{n-1}/4$. Let this new tree be $T_n$ with embedded terminal set $N_n$. By
construction, $T_n$ has the correct topology $G_n$.

Let $T_n'$ be any Steiner tree (but not necessarily an SMT) on $N_n$ with topology \emph{not} $G_n$. Suppose we collapse $\mathbf{t}_n$ and
$\mathbf{t}_{n-1}$ onto the point $\mathbf{s}$ (fixing all other nodes in the network), and consider the topology $G$ of the resulting network.
$G$ is a topology on $n-1$ terminals, but may be different from $G_{n-1}$, indeed $G$ is not necessarily a tree. If $G=G_{n-1}$, then $T_n'$
also has the same topology as $T_n$, which by convexity and the fact that $T_n'$ is a Steiner tree implies that $T_n'=T_n$ (see Theorem 1.3 of
\cite{bib14}); this is a contradiction and hence $G\not=G_{n-1}$. It follows from this that, if we consider the network $T_n' \cup
\{\mathbf{st}_n\}$ (which interconnects $N_{n-1}$), we have $L(T_n') + \vert \mathbf{st}_n \vert \geq L(T_{n-1})+ f_{n-1}$. This implies that
\begin{eqnarray*}
  L(T_n') & \geq & L(T_{n-1})+ f_{n-1}- \vert \mathbf{st}_n \vert \\
   &=&  L(T_{n-1})+ 3\vert \mathbf{st}_n \vert = L(T_{n})+ \vert \mathbf{st}_n \vert.
\end{eqnarray*}
Hence, we can choose $f_n= f_{n-1}/4 <   L(T_n') - L(T_n)$.

The claim (and lemma) now follow. Furthermore, the iterative algorithm for constructing a suitable set of embedded terminals for any required
Steiner topology is constructive with $f_i= (3\sqrt{3}-5)/4^{i-4}$ for each $i \geq 4$. \noindent
\end{proof}

\begin{proposition}Let $G_n$ be a full Steiner topology on $n$ terminals. There exists an embedded set of terminals $N$ in the Euclidean plane
such that $\mathrm{beads}(T_S)=\mathrm{beads}(T_{\mathrm{opt}})+2n-4$ and $T_S$ has topology $G_n$.
\end{proposition}

\begin{proof}
We construct an SMT $T_S$ with topology $G_n$ by repeatedly sprouting terminals, starting from a full Steiner tree on three terminals called the
\textit{base}. By the previous proposition any full Steiner topology can be produced in this way. Note that we can create a base with edges of
any length by simply intersecting the end-points of three line segments at one common point such that every pair of segments forms an angle of
$120^\circ$ (and we have complete freedom to do this since we are constructing an SMT by choosing positions for the terminals). By making the
edges of the base large enough, it is clear that we can construct $T_S$ such that every edge-length has the form $a_i \pm \varepsilon_i$, where
$a_i$ is an integer of order at least two and $\varepsilon_i$ has any predefined value between zero and one. $T_S$ is then converted into an
MSPT by a sequence of displacements (which we describe below) of the Steiner points, where displacements do not change the original topology
$G_n$.

In $T_S$, let $\mathbf{s}_0$ be a Steiner point adjacent to a terminal $\mathbf{t}$ and two other nodes $\mathbf{v}_1,\mathbf{v}_2$ where
edge-lengths are preselected as follows: $\vert \mathbf{ts}_0 \vert=a_1-\varepsilon$, $\vert \mathbf{s}_0\mathbf{v}_1 \vert = \vert
\mathbf{s}_0\mathbf{v}_2 \vert = b_1+\varepsilon_1$ for large integers $a_1,b_1$ and $0<\varepsilon,\varepsilon_1<1$. In the first step (Figure
\ref{figureS1Perturb}) we displace $\mathbf{s}_0$ along the line through $\mathbf{t}$ and $\mathbf{s}_0$ and in the direction of the vector
$\overrightarrow{\mathbf{ts}_0}$. We displace until $\vert \mathbf{ts}_0 \vert=a_1-\varepsilon^\prime$ and $\vert \mathbf{s}_0\mathbf{v}_1 \vert
= \vert \mathbf{s}_0\mathbf{v}_2 \vert = b_1-\varepsilon_1^\prime$ for some $0<\varepsilon^\prime,\varepsilon_1^\prime<1$. Clearly this is
possible as long as we preselect $\varepsilon_1$ to be small enough compared to $\varepsilon$.

\begin{figure}[htb]
\begin{center}

\includegraphics[scale=0.4]{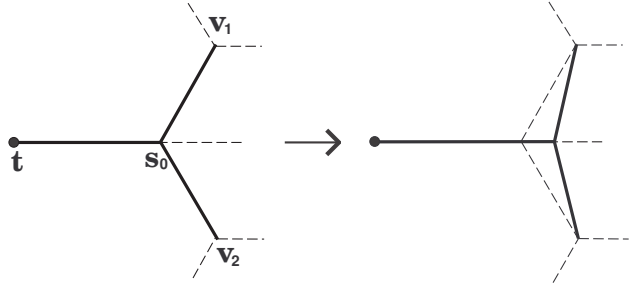}

\end{center}
\caption{First step of the displacement sequence.  \label{figureS1Perturb}}
\end{figure}

We now displace all other Steiner points in a depth-first or breadth-first order rooted at $\mathbf{t}$. Suppose that in the process we have
reached the Steiner point $\mathbf{s}$ with parent $\mathbf{s}^\prime$ and children $\mathbf{u}_1,\mathbf{u}_2$. We displace $\mathbf{s}$ along
the line through $\mathbf{s}$ and the point $\mathbf{p}$ and in the direction $\overrightarrow{\mathbf{ps}}$, where $\mathbf{p}$ is the position
$\mathbf{s}^\prime$ had \textit{before} its displacement; see Figure \ref{figurePerturb}. If $\vert \mathbf{ss}^\prime \vert =a-\varepsilon_1$
then we preselect $\vert \mathbf{su}_1 \vert = \vert \mathbf{su}_2 \vert =b+\varepsilon_2$ for $0<\varepsilon_2<1$. We select $\varepsilon_2$
small enough so that the displacement of $\mathbf{s}$ produces the lengths $\vert \mathbf{ss}^\prime \vert =a-\varepsilon_1^\prime$ and $\vert
\mathbf{su}_1 \vert = \vert \mathbf{su}_2 \vert =b-\varepsilon_2^\prime$, for some $0<\varepsilon_1^\prime,\varepsilon_2^\prime<1$. We continue
this process until we have displaced all Steiner points. Call the resultant tree $T$. Note that the edges of $T_S$ were preselected so that one
edge has length $a_1-\varepsilon$ and every other edge $e_i$ has length $b_i+\varepsilon_i$. After all displacements the first edge has length
$a_1-\varepsilon^\prime$ and every other $e_i$ has length $b_i-\varepsilon_i^\prime$. Clearly then $\mathrm{beads}(T_S)=\mathrm{beads}(T)+2n-4$
and $T$ is an MSPT.
\end{proof}

\begin{figure}[htb]
\begin{center}

\includegraphics[scale=0.4]{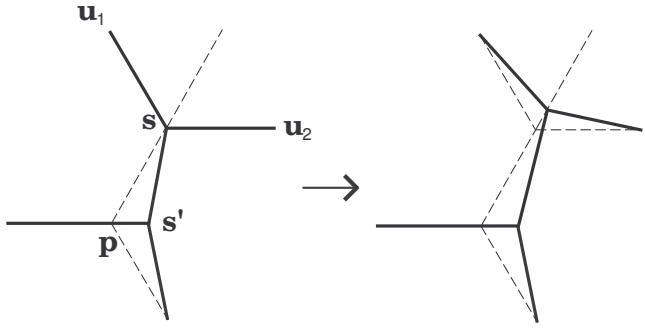}

\end{center}
\caption{General step of the displacement sequence.  \label{figurePerturb}}
\end{figure}

\section{Minkowski Planes} \label{sec3}
A \textit{Minkowski plane} is a two-dimensional real normed space $M=(\mathbb{R}^2,||\cdot||_M)$ with \textit{unit ball}
$B=\{\mathbf{x}:||\mathbf{x}||_M\leq1\}$. We denote the metric induced by $M$ by $d_M(\mathbf{x},\mathbf{y})=||\mathbf{x}-\mathbf{y}||_M$.
Examples of Minkowski planes include the Euclidean plane and the rectilinear plane, where the unit balls are the circle and the $45^\circ$
rotated square respectively. The unit ball of a Minkowski plane is always convex, centrally symmetric and bounded in the Euclidean norm.
Conversely, any such convex body is the unit ball of a Minkowski plane. The boundary of a ball $B$ is denoted by $\mathrm{bd}(B)$ and its
interior by $\mathrm{int}(B)$.

The question arises as to whether the upper bound from Proposition \ref{mainUpperProp} is best possible in all Minkowski planes. In the
three-terminal case we show that the upper bound can be improved for a given Minkowski plane if and only if the unit ball is a parallelogram.
Well-known Minkowski planes with this property are the $L_1$ (rectilinear) and $L_\infty$ planes.

Let $N=\{\mathbf{t}_i:i=1,2,3\}$ be a set of embedded terminals and let $T_S$ be a $d_M$-SMT on $N$. Many of the propositions below will refer
to $T_S^*$ instead of $T_S$ in order to maintain generality. Recall that this convention may lead to zero-length edges, and consequently also to
balls of zero radius.

Throughout this section we let $L$ denote a Minkowski plane with a parallelogram unit ball. The corresponding metric is denoted by $d_L$. If
$\mathbf{t}$ is a terminal point in the plane, we denote by $l_1(\mathbf{t})$ and $l_2(\mathbf{t})$ the Euclidean straight lines passing through
$\mathbf{t}$ and parallel to the major diagonal and minor diagonal, respectively, of the parallelogram defining the unit ball of $L$. Minkowski
balls $\{B_i:i=1,2,3\}$ \textit{tessellate} at $\mathbf{x}$ if $B_i \cap B_j$ is a point or a Euclidean line segment whenever $i\neq j$, and
$\bigcap B_i=\{\mathbf{x}\}$. The point $\mathbf{x}$ is called a \textit{tessellation point} of $N$ if there exists a set of balls
$\{B_i:i=1,2,3\}$, with $B_i$ centered at $\mathbf{t}_i$, that tessellate at $\mathbf{x}$. The next proposition is a generalization of a
well-known result (\textit{cf}. \cite{bib11}) on three-terminal rectilinear SMTs.

\begin{proposition}\label{tessellate}Let $N=\{\mathbf{t}_i:i=1,2,3\}$ be a set of embedded terminals. Then there exists a $d_L$-SMT $T_S$ on $N$
such that the Steiner point of $T_S^*$ coincides with the intersection of the median of $\{l_1(\mathbf{t}_i)\}$ and the median of
$\{l_2(\mathbf{t}_i)\}$.
\end{proposition}
\begin{proof}
Let $\mathbf{x}$ be a point in the plane. We wish to minimize the function $f=|\mathbf{t}_1\mathbf{x}|+|\mathbf{t}_2\mathbf{x}|+
|\mathbf{t}_3\mathbf{x}|$ where all inequalities $|\mathbf{t}_i\mathbf{x}|+|\mathbf{t}_j\mathbf{x}| \geq |\mathbf{t}_i\mathbf{t}_j|$, with
$1\leq i<j\leq 3$, hold by the triangle inequality. Therefore a minimum would occur if $|\mathbf{t}_i\mathbf{x}|+|\mathbf{t}_j\mathbf{x}| =
|\mathbf{t}_i\mathbf{t}_j|$ for every $1\leq i<j\leq 3$; equivalently, a minimum would occur if the balls $\{B_i:i=1,2,3\}$, with $B_i$ centered
at $\mathbf{t}_i$ and of radius $|\mathbf{t}_i\mathbf{x}|$, tessellate at $\mathbf{x}$. We show that this happens if $\mathbf{x}$ is the
intersection of the median of $\{l_1(\mathbf{t}_i)\}$ and the median of $\{l_2(\mathbf{t}_i)\}$.

Suppose, without loss of generality, that $l_1(\mathbf{t}_2)$ is the median of $\{l_1(\mathbf{t}_i)\}$ and $l_2(\mathbf{t}_3)$ is the median of
$\{l_2(\mathbf{t}_i)\}$, and let $\mathbf{x}$ be the intersection of these two lines. Let $l_0$ be the line through $\mathbf{x}$ and parallel to
the side of $B_3$ that intersects $B_2$ at $\mathbf{x}$ only - see Figure \ref{figureMedians}. Then $\mathbf{t}_1$ must lie on the opposite side
of $l_0$ to $\mathbf{t}_2$ and $\mathbf{t}_3$ (since $l_1(\mathbf{t}_2)$ and $l_2(\mathbf{t}_3)$ are medians). Therefore $B_1 \cap B_2 \subset
l_0$ and $B_1 \cap B_3 \subset l_0$ and the result follows.
\end{proof}

\begin{figure}[htb]
\begin{center}

\includegraphics[scale=0.4]{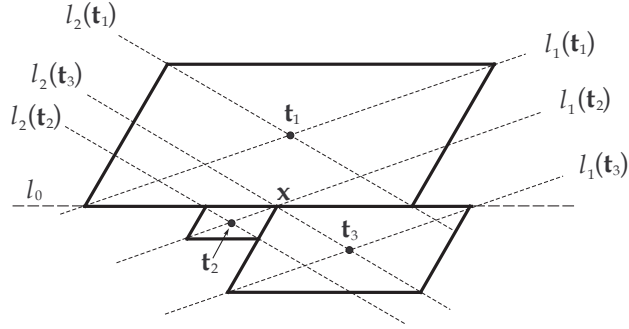}

\end{center}
\caption{Median diagonals $l_1(\mathbf{t}_2)$ and $l_2(\mathbf{t}_3)$ intersect at Steiner point $\mathbf{x}$. \label{figureMedians}}
\end{figure}

\begin{corollary}\label{corTessellate}Let $N=\{\mathbf{t}_i:i=1,2,3\}$ be a set of embedded terminals and let $T_S$ be a $d_L$-SMT on $N$.
Then the Steiner point of $T_S^*$ is the unique tessellation point of $N$.
\end{corollary}
\begin{proof}
The set of linear equations $|\mathbf{t}_i\mathbf{x}|+|\mathbf{t}_j\mathbf{x}| = |\mathbf{t}_i\mathbf{t}_j|$, for $1\leq i < j \leq 3$, has a
unique solution.
\end{proof}

Given three terminals in the plane and a unit ball parallelogram $B$, the \textit{enclosing diagonalized parallelogram} is the smallest
parallelogram whose sides are parallel to the major and minor diagonals of $B$ and which includes all the terminals on its boundary. The next
result also has an analogue in the rectilinear plane (\textit{cf}. \cite{bib11}).

\begin{corollary}The total length of a three-terminal $d_L$-SMT is equal to half the perimeter of the enclosing diagonalized parallelogram.
\end{corollary}

The following proposition is the first main result of this section. It shows that the upper bound from Proposition \ref{mainUpperProp} with
$n=3$ is not strict for parallelogram-based Minkowski planes.

\begin{proposition}\label{propPara}Let $T_S$ be a $d_L$-SMT and let $T_{\mathrm{opt}}$ be a $d_L$-MSPT on the embedded terminals $\{\mathbf{t}_i:i=1,2,3\}$.
Then $\mathrm{beads}(T_S)\leq \mathrm{beads}(T_{\mathrm{opt}})+1$.
\end{proposition}

\begin{proof}
Suppose that the Steiner point of $T_S^*$ is $\mathbf{s}$ and its edges are $a_i=\mathbf{t}_i\mathbf{s}$ for $i \in \{1,2,3\}$. Let
$\mathbf{s^\prime}$ be the Steiner bead of $T_{\mathrm{opt}}^*$ and let $e_i=\mathbf{t}_i\mathbf{s^\prime}$ for $i \in \{1,2,3\}$. Since
$\mathbf{s}$ is the tessellation point of $\{\mathbf{t}_i\}$, at most one inequality from $\vert e_i\vert < \vert a_i \vert$, $i\in \{1,2,3\}$
can be true (note that if none of these inequalities are true then $\mathrm{beads}(T_S)= \mathrm{beads}(T_{\mathrm{opt}})$, and we are done).
Suppose w.l.o.g that $\vert e_1 \vert< \vert a_1 \vert$. Let $p_1=\vert a_1 \vert - \vert e_1\vert$ and let $p_i=\vert e_i \vert - \vert a_i
\vert$ for $i=\{2,3\}$.

\noindent\textbf{Claim}: $p_1 \leq \min{\{p_2,p_3 \}}$.\\
Clearly $\vert a_1 \vert + \vert a_3\vert=\vert \mathbf{t}_1\mathbf{t}_3\vert$ since $\mathbf{t}_1,\mathbf{s},\mathbf{t}_3$ is a shortest path
between $\mathbf{t}_1$ and $\mathbf{t}_3$. Also, by using the triangle inequality in $\triangle \mathbf{t}_1\mathbf{s}^\prime\mathbf{t}_3$ we
obtain $\vert \mathbf{t}_1\mathbf{t}_3\vert\leq \vert e_1\vert+\vert e_3\vert$. Therefore $\vert a_1\vert+\vert a_3\vert \leq (\vert
a_1\vert-p_1)+(\vert a_3\vert+p_3)$ so that $p_1\leq p_3$. Similarly, $p_1 \leq p_2$ and this proves the claim.

We now have:
\begin{eqnarray*}
\mathrm{beads}(T_S)- \mathrm{beads}(T_{\mathrm{opt}}) &=& \sum\limits_{i=1}^3\lceil\vert a_i \vert\rceil-\sum\limits_{i=1}^3\lceil\vert e_i \vert\rceil\\
&=& \lceil\vert a_1 \vert\rceil - \lceil\vert e_1\vert\rceil-\sum\limits_{i\in\{2,3\}}\{\lceil \vert e_i \vert\rceil - \lceil \vert a_i\vert\rceil \}\\
&=& \lceil\vert e_1 \vert+p_1\rceil - \lceil\vert e_1\vert\rceil-\sum\limits_{i\in\{2,3\}}\{\lceil \vert e_i \vert\rceil -
\lceil \vert e_i\vert-p_i\rceil \}\\
&\leq & \lceil p_1 \rceil -\sum\limits_{i\in\{2,3\}}\{\lceil p_i \rceil -1\}\\
&\leq & 0-\lceil \max\{p_2,p_3\} \rceil +2 \hspace{0.5 in} \mathrm{(since \ } 0<p_1 \leq \min{\{p_2,p_3 \}}\mathrm{)}\\
&<& 2.
\end{eqnarray*}
\end{proof}

To prove a converse of the previous proposition we first show that Corollary \ref{corTessellate} is unique to Minkowski planes with
parallelogram unit balls. We find that some three-terminal sets in Minkowski planes with hexagon unit balls do have tessellation points, but
that this is not true in general. For any non-parallelogram based Minkowski plane we then construct a three-terminal example that achieves the
upper bound from Proposition \ref{mainUpperProp}.

Two points on the boundary of a ball $B$ form a \textit{diametric pair} if the Euclidean straight line passing through the points also passes
through the center of $B$. A point $\mathbf{z}$ on a ball $B_i$ is \textit{equivalent} to a point $\mathbf{z}^\prime$ on a ball $B_j$ if and
only if, by translating $B_i$ so that its center coincides with the center of $B_j$, it is possible to rescale $B_i$ so that $\mathbf{z}$
coincides with $\mathbf{z}^\prime$. Let $\{\mathbf{t}_i:i=1,2,3\}$ be a set of embedded terminals with tessellation point $\mathbf{s}$ in an
arbitrary non-parallelogram-based Minkowski plane $M$, and suppose that the balls $\{B_i\}$ tessellate at $\mathbf{s}$.

\begin{lemma}Suppose that $B_i \cap B_j=\{\mathbf{s}\}$ for some $i,j \in \{1,2,3\}$. Then the point that forms a diametric pair with
$\mathbf{s}$ on $B_i$ is equivalent to $\mathbf{s}$ when considered as a point on $B_j$.
\end{lemma}
\begin{proof}
This follows from the convexity and central symmetry of the balls.
\end{proof}

As a consequence of the previous lemma there exist distinct $i,j,k \in \{1,2,3\}$ such that $B_i \cap B_j$ and $B_i \cap B_k$ are Euclidean line
segments (as opposed to single points only corresponding to $\mathbf{s}$). Suppose w.l.o.g that $B_1 \cap B_2$ and $B_1 \cap B_3$ are line
segments.

\begin{lemma}$B_2 \cap B_3$ is also a line segment.
\end{lemma}
\begin{proof}
With the aim of producing a contradiction we assume that $B_2 \cap B_3=\{\mathbf{s}\}$. Note that there exist exactly four maximal-length line
segments in $\bigcup \{\mathrm{bd}\left(B_i\right)\}$ that have $\mathbf{s}$ as an endpoint. We list these segments in any clockwise order as
$\{S_i:0 \leq i \leq 3\}$. Since $\mathbf{s}$ forms a diametric pair on $B_2$ and $B_3$, $S_i$ and $S_{(i+2)\mathrm{mod}4}$ have the same
gradient for any $i \in \{0,...,3\}$. By central symmetry, the only possible balls that can produce such a configuration of line segments are
parallelograms, which is a contradiction.
\end{proof}

A point $\mathbf{x}$ on $\mathrm{bd}(B)$ is called a \textit{corner} if the intersection of any neighborhood of $\mathbf{x}$ with
$\mathrm{bd}(B)$ is not a Euclidean straight line segment.

\begin{lemma}The point $\mathbf{s}$ is a corner of every member of $\{B_i\}$.
\end{lemma}
\begin{proof}
Clearly $\mathbf{s}$ is a corner of at least two members of $\{B_i\}$. Suppose, for a contradiction, that $\mathbf{s}$ is not a corner of some
$B_i$. In this case there exist exactly three maximal-length line segments in $\bigcup \{\mathrm{bd}\left(B_i\right)\}$ that have $\mathbf{s}$
as an endpoint. Furthermore, exactly two of these segments have the same gradient. By central symmetry such a configuration of line segments can
only be produced by parallelogram balls.
\end{proof}

By combining the previous lemmas, we know there exist exactly three maximal-length line segments in $\bigcup \{\mathrm{bd}\left(B_i\right)\}$
that have $\mathbf{s}$ as an endpoint, all with distinct gradients. From this fact and central symmetry we conclude the following lemma:

\begin{lemma}\label{hexagon}The balls $\{B_i\}$ are hexagons.
\end{lemma}

We also need the following two lemmas which follow directly from results by Martini, Swanepoel and Weiss \cite{bib12}.

\begin{lemma}Suppose that $\mathbf{s}$ is a degree-three Steiner point of a $d_M$-SMT on embedded terminal set $\{\mathbf{t}_i:i=1,2,3\}$, where
$\mathbf{s}$ does not coincide with a terminal. Then $\mathbf{s}$ is also a Steiner point of the terminal set $\{\mathbf{t}_i^\prime\}$ where
$\mathbf{t}_i^\prime$ is any point lying on the Euclidean ray with origin $s$ and passing through $\mathbf{t}_i$.
\end{lemma}

\begin{lemma}There exists a set $\{\mathbf{t}_i:i=1,2,3\}$ of embedded terminals such that some $d_M$-SMT on $\{\mathbf{t}_i\}$ has a degree-three
Steiner point that does not coincide with a terminal.
\end{lemma}

We can now prove the final proposition of this section. If the unit ball defining $M$ is not a hexagon then, by using the previous two lemmas,
we can construct a \textit{critical} $d_M$-SMT. A critical $d_M$-SMT on an embedded terminal set $\{\mathbf{t}_i:i=1,2,3\}$ has the following
properties:
\begin{enumerate}
    \item The Steiner point does not coincide with a terminal: i.e., the Steiner point is of degree three and there are no edges of zero length,
    \item Each edge $e_i$ has length $a_i+\varepsilon_i$ where $a_i$ is an integer and $0<\varepsilon_i<1$ has any predefined value,
    \item The balls centered at the terminals and meeting the Steiner point do not tessellate.
\end{enumerate}

Otherwise, if $M$'s unit ball is a hexagon we first use the previous two lemmas to find a terminal set satisfying properties (1) and (2). We
then destroy the tessellation property (if necessary) by performing a rotational displacement around the Steiner point of one of the terminals.
Let $T_S$ be a critical $d_M$-SMT on $\{\mathbf{t}_i\}$. The final result of this section is a converse to Proposition \ref{propPara}.

\begin{proposition}$\mathrm{beads}(T_S)= \mathrm{beads}(T_{\mathrm{opt}})+2$.
\end{proposition}
\begin{proof}
By Properties (1) and (3) we may displace the Steiner point $\mathbf{s}$ into a region corresponding to the intersection of the interiors of two
balls, say $\mathrm{int}(B_1) \cap \mathrm{int}(B_2)$. By Property (2) we can preselect each $\varepsilon_i$ so that after displacement we have
$\vert \mathbf{t}_1\mathbf{s} \vert< a_1$, $\vert \mathbf{t}_2\mathbf{s} \vert< a_2$ and $\vert \mathbf{t}_3\mathbf{s} \vert \leq a_3+1$.
\end{proof}

We conclude this section with two conjectures. Let $N$ be a set of $n$ terminals in a Minkowski plane $M$ with unit ball $B$, and let $T_S$ and
$T_{\mathrm{opt}}$ be a $d_M$-SMT and $d_M$-MSPT on $N$ respectively. Suppose first that $B$ is a parallelogram. Let $\mathbf{s}^\prime$ be the
Steiner point of some cherry of $T_S$, and let $\mathbf{u}_1,\mathbf{u}_2$ be terminals adjacent to $\mathbf{s}^\prime$. The proof of
Proposition \ref{propPara} implies that a displacement of $\mathbf{s}^\prime$ can shorten at most one of
$\mathbf{s}^\prime\mathbf{u}_1,\mathbf{s}^\prime\mathbf{u}_2$. Since every Steiner tree has a minimum of two cherries, displacements of Steiner
points can shorten at most $2n-5$ edges of $T_S$ (note, of course, that it may be possible to shorten up to $2n-4$ edges of $T_S$ if we change
its topology).

\begin{conjecture}The upper bound $\mathrm{beads}(T_S)-\mathrm{beads}(T_{\mathrm{opt}}) \leq 2n-4$ is tight if and only if $B$ is not a parallelogram.
\end{conjecture}

\begin{conjecture}The upper bound $\mathrm{beads}(T_S)-\mathrm{beads}(T_{\mathrm{opt}}) \leq 2n-5$ is tight if and only if $B$ is a parallelogram.
\end{conjecture}

\section{A Canonical Form for Euclidean MSPTs} \label{sec4}
Throughout this section we only consider Euclidean MSPTs. In general there are many possible ways to embed an MSPT in Euclidean space. Here we
introduce a canonical form for MSPTs (over all possible embeddings) which allows us to reformulate the MSPT problem as that of finding a
shortest total length tree in which almost all edges have integer length. Understanding this canonical form provides a valuable first step
towards finding an efficient exact algorithm for the problem (like the canonical forms for Steiner trees in fixed-orientation metrics
\cite{bib16}, and those used in the previously mentioned GeoSteiner algorithms for rectilinear Steiner trees). In other words, we show that in
order to find MSPTs, it suffices to explore a class of trees with strong structural restrictions. The canonical form we describe is also
interesting from a more theoretical point of view as it gives an insight into the geometry of MSPTs.

An MSPT where every terminal is of degree one and every Steiner bead is of degree three is called a \textit{full MSPT}. This term refers to
MSPTs that have this property ``naturally", i.e., not through a splitting process. Recall that a full MSPT contains $n-2$ Steiner beads and
$2n-3$ edges. The Steiner bead of a cherry will be referred to as a \textit{cherry bead}. A \textit{level-region} of a function $f$ is a set of
points satisfying $f=k$ for some constant $k$.

\begin{lemma}Let $N$ be a set of three embedded terminals admitting a full MSPT. Then there exists an MSPT on $N$ such that at least two of its
edges are of integer length.
\end{lemma}
\begin{proof}
Let $T_{\mathrm{opt}}$ be a full MSPT on the embedded terminal set $N=\{\mathbf{t}_i:i=1,2,3\}$ and let $\mathbf{s}$ be the Steiner bead. Then
$\mathrm{beads}(T_{\mathrm{opt}})$ is the minimum value of the function $f=\lceil \vert \mathbf{t}_1\mathbf{s} \vert \rceil + \lceil \vert
\mathbf{t}_2\mathbf{s} \vert \rceil + \lceil \vert \mathbf{t}_3\mathbf{s} \vert \rceil - 2$, where the position of $\mathbf{s}$ is variable. The
level-region $k=\lceil \vert \mathbf{t}_1\mathbf{s} \vert \rceil + \lceil \vert \mathbf{t}_2\mathbf{s} \vert \rceil + \lceil \vert
\mathbf{t}_3\mathbf{s} \vert \rceil - 2$, where $k$ is a positive integer, consists of regions that are bounded by at least one and at most six
integer-radius circular arcs. Since $\mathbf{s}$ does not correspond to a terminal, the region $L(\mathbf{s})$ containing $\mathbf{s}$ must be
bounded by at least two arcs. Displacing $\mathbf{s}$ to coincide with an intersection point of the arcs bounding $L(\mathbf{s})$ will lead to
an MSPT of the desired form; see Figure \ref{figureLevs}.
\end{proof}

\begin{figure}[htb]
\begin{center}

\includegraphics[scale=0.3]{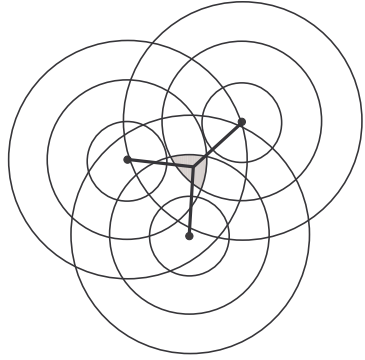}

\end{center}
\caption{Level-regions of $f$. \label{figureLevs}}
\end{figure}

\begin{definition}A tree $T$ connecting $n$ embedded terminals and some Steiner points is called $\mathbb{Z}$-packed if $T$ has a full Steiner
topology and at least $2n-4$ of its edges are of integer length.
\end{definition}

\begin{proposition}Let $N$ be a set of three embedded terminals in the Euclidean plane admitting a full MSPT. Then a shortest total length
$\mathbb{Z}$-packed tree on $N$ is an MSPT.
\end{proposition}
\begin{proof}
By the previous lemma, $N$ must admit a $\mathbb{Z}$-packed MSPT. Note that Proposition \ref{mainUpperProp} still holds if we modify it slightly
by constraining $T_{\mathrm{opt}}$ to be $\mathbb{Z}$-packed and by letting $T_S$ be a shortest $\mathbb{Z}$-packed tree. Corollary
\ref{EqCorollary} now gives us our result.
\end{proof}

We wish to generalize the previous result to any number of terminal points, but for this we need another condition. If two edges of an MSPT are
incident to the same Steiner bead and are collinear then these edges are said to form a \textit{Steiner bond}. An MSPT $T$ on a set $N$ of
embedded terminals is called \textit{bond-free} if every MSPT on $N$ with the same topology as $T$ is free of Steiner bonds. Figure
\ref{figureBond} provides an example of an MSPT that is not bond-free; the fact that the depicted tree is an MSPT on the three solid nodes
follows from Proposition \ref{mainUpperProp} once it is noted that the SMT on the same terminals has $4$ beads.

\begin{figure}[htb]
\begin{center}

\includegraphics[scale=0.5]{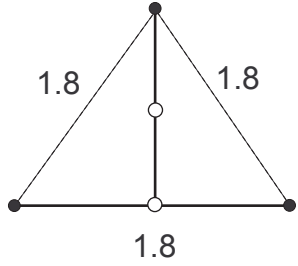}

\end{center}
\caption{An MSPT that is not bond-free.  \label{figureBond}}
\end{figure}

\begin{lemma}Let $N$ be a set of four embedded terminals admitting a bond-free full MSPT $T_{\mathrm{opt}}$. Then it is possible to find an MSPT on $N$,
with the same topology as $T_{\mathrm{opt}}$, such that all edges incident to terminals are of integer length.\label{FourPRep}
\end{lemma}
\begin{proof}
Let $N=\{\mathbf{t}_i:i=1,..,4 \}$ be the terminal set for $T_{\mathrm{opt}}$ and let $\mathbf{s}_1,\mathbf{s}_2$ be the Steiner beads of
$T_{\mathrm{opt}}$ with $\mathbf{s}_1$ adjacent to $\mathbf{t}_1$ and $\mathbf{t}_2$. Let $T_0$ be the subtree of $T_{\mathrm{opt}}$ induced by
the nodes $\mathbf{t}_1,\mathbf{t}_2,\mathbf{s}_1,\mathbf{s}_2$. Clearly $T_0$ is an MSPT on the nodes $\mathbf{t}_1,\mathbf{t}_2,\mathbf{s}_2$.
By fixing the position of $\mathbf{s}_2$ we convert $T_0$ into a $\mathbb{Z}$-packed tree by displacing $\mathbf{s}_1$. We now fix
$\mathbf{s}_1$ at its new position and displace $\mathbf{s}_2$ until the subtree induced by
$\mathbf{t}_3,\mathbf{t}_4,\mathbf{s}_1,\mathbf{s}_2$ is $\mathbb{Z}$-packed. The modified $T_{\mathrm{opt}}$ must now have at least three
integer length edges. Suppose that the edges $\mathbf{t}_1\mathbf{s}_1,\mathbf{s}_1\mathbf{s}_2,\mathbf{t}_3\mathbf{s}_2$ have integer lengths
(the other cases are handled similarly). We fix $\mathbf{s}_2$ and displace $\mathbf{s}_1$ along the circle centered at $\mathbf{t}_1$ and of
radius $\vert \mathbf{t}_1\mathbf{s}_1 \vert$. Note that the smallest value of $\vert \mathbf{s}_1\mathbf{s}_2 \vert$ occurs when $\mathbf{s}_1$
is displaced until it reaches the line connecting $\mathbf{t}_1$ and $\mathbf{s}_2$. Therefore displacement until $\vert
\mathbf{t}_2\mathbf{s}_1 \vert$ is an integer is possible due to the bond-free condition; see Figure \ref{figureFourP}. We fix the position of
$\mathbf{s}_1$ and repeat the process for $\mathbf{s}_2$.
\end{proof}

\begin{figure}[htb]
\begin{center}

\includegraphics[scale=0.5]{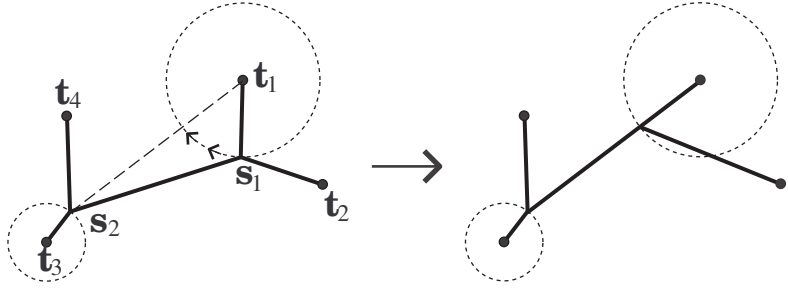}

\end{center}
\caption{Bond creation.  \label{figureFourP}}
\end{figure}

A \textit{caterpillar} is a tree with the property that the removal of its degree-one nodes results in a path.

\begin{lemma}Let $N$ be a set of $n \geq 4$ embedded terminals admitting a bond-free full MSPT $T_{\mathrm{opt}}$ and suppose that $T_{\mathrm{opt}}$
is a caterpillar. Let $\mathbf{s}_0$ be a cherry bead of $T_{\mathrm{opt}}$ connected to another Steiner bead $\mathbf{s}_1$. If
$T_{\mathrm{opt}}$ is bond-free then there exists a $\mathbb{Z}$-packed MSPT on $N$ with the same topology as $T_{\mathrm{opt}}$ such that every
edge, other than possibly $\mathbf{s}_0\mathbf{s}_1$, has integer length.
\end{lemma}
\begin{proof}
This follows readily from repeated application of the previous lemma.
\end{proof}

Let $T$ be a non-caterpillar full MSPT rooted at two terminals connected to a cherry bead $\mathbf{r}$, and let $\mathbf{s}$ be any Steiner bead
of $T$ with children $\mathbf{v}_1,\mathbf{v}_2$. Then $\mathbf{s}$ is called a \textit{junction} of $T$ if, for each $i\in\{1,2\}$, the subtree
induced by $\mathbf{s}$,$\mathbf{v}_i$ and all descendants of $\mathbf{v}_i$ (if they exist) is a caterpillar. A Steiner bead is a
\textit{maximal} junction if it is a junction but its parent is not a junction - see Figure \ref{figureJunc}.

\begin{figure}[htb]
\begin{center}

\includegraphics[scale=0.5]{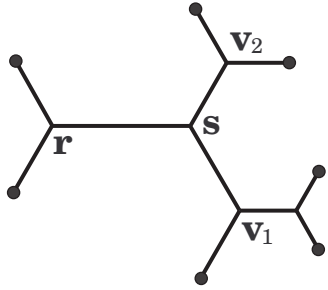}

\end{center}
\caption{Maximal junction $\mathbf{s}$. \label{figureJunc}}
\end{figure}

\begin{lemma}Let $N$ be a set of $n \geq 4$ embedded terminals admitting a bond-free full MSPT. Then it is possible to find an MSPT, say $T$, on
$N$ such that at most one of its edges are of non-integer length. Furthermore, either all $T$'s edges will be of integer length, or we will be
able to choose which internal edge of $T$ has non-integer length.
\end{lemma}
\begin{proof}
Root $T_{\mathrm{opt}}$ at two terminals connected to a cherry bead $\mathbf{r}$. Let $\mathbf{s}$ be a maximal junction of $T_{\mathrm{opt}}$
with children $\mathbf{v}_1,\mathbf{v}_2$ and for each $i \in \{1,2\}$ let $T_i$ be the subtree of $T_{\mathrm{opt}}$ induced by
$\mathbf{s}$,$\mathbf{v}_1$,$\mathbf{v}_2$, the parent of $\mathbf{s}$ and the descendants of $\mathbf{v}_i$. Note that $T_i$ is a caterpillar.
By fixing the positions of the parent of $\mathbf{s}$ and the child of $\mathbf{s}$ not equal to $\mathbf{v}_i$, and applying the previous
lemma, we force all edges of $T_i$ (except possibly $\mathbf{sv}_i$) to be of integer length. We do this for all maximal junctions of
$T_{\mathrm{opt}}$. For the resultant tree, say $T$, we now consider $\mathbf{v}_1$ and $\mathbf{v}_2$ as terminals and ignore the subtrees
induced by the descendants of $\mathbf{v}_1$ and $\mathbf{v}_2$ (similarly for other junctions). We select the maximum junctions with respect to
$T$ and continue the process. Once there are no junctions left (i.e., only a caterpillar remains) we force all edges, except possibly the edge
between $\mathbf{r}$ and its child, to be of integer length. Repeated application of Lemma \ref{FourPRep} now allows us to choose which internal
edge should possibly be of non-integer length, i.e., we now ``move" the non-integer property to any other internal edge.
\end{proof}

\begin{proposition}Let $N$ be any set of terminals in the Euclidean plane admitting a bond-free full MSPT $T_{\mathrm{opt}}$. Then a shortest
total length $\mathbb{Z}$-packed tree on $N$ is an MSPT.
\end{proposition}
\begin{proof}
This follows, as before, from the previous result and from Corollary \ref{EqCorollary} after a slight modification to Proposition
\ref{mainUpperProp}.
\end{proof}

\section{Concluding Remarks and Conjectures}
We suspect that, at least in the Euclidean case, the performance difference $2n-4$ of the SMT heuristic can be improved by supplementing it with
an algorithm that involves relatively small displacements of the Steiner points. The question is: by how much can we improve the performance? If
for any set of embedded terminals it is possible to find an MSPT with the same topology (or a degeneracy thereof) as an SMT on the terminals,
then it would be theoretically possible to improve the performance to optimality in this way. If an SMT $T_S$ is not full then certainly the
topology of an MSPT on the same terminals is generally not a degeneracy of $T_S$ (i.e., an MSPT topology cannot be obtained simply by collapsing
edges of $T_S$). Consider for instance the embedded terminals $\{\mathbf{t}_i:i=1,2,3\}$ in the Euclidean plane where $\angle
\mathbf{t}_1\mathbf{t}_2\mathbf{t}_3=120^\circ$ and $\vert \mathbf{t}_1\mathbf{t}_2\vert=\vert \mathbf{t}_2\mathbf{t}_3\vert=5.1$ units. Clearly
the SMT $T_S$ on $\{\mathbf{t}_i\}$ is not full and $\mathrm{beads}(T_S)=10$. However, $\mathrm{beads}(T_{\mathrm{opt}})=9$ as shown in Figure
\ref{figureTop}.

\begin{figure}[htb]
\begin{center}

\includegraphics[scale=0.5]{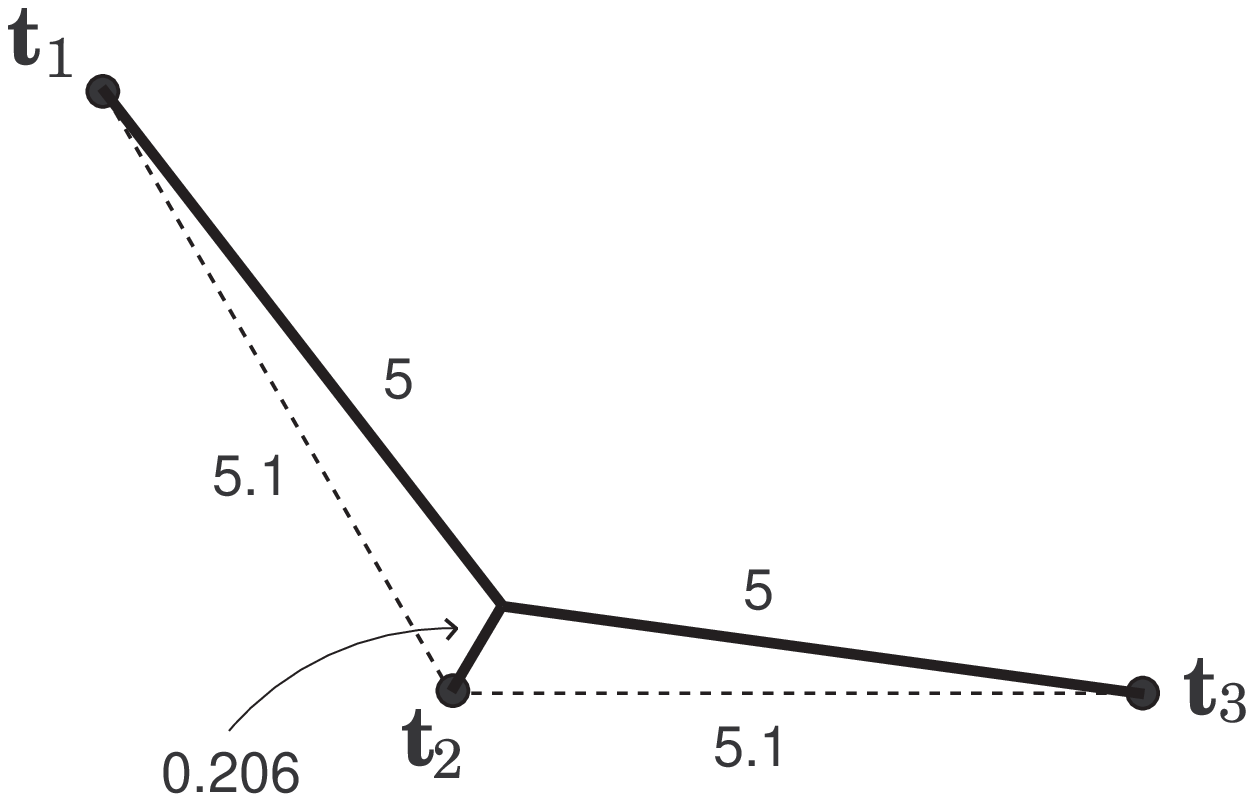}

\end{center}
\caption{$T_{\mathrm{opt}}$ with nine beads. \label{figureTop}}
\end{figure}

To conclude we now state some conjectures that we hope will inspire further research into the relationship between SMTs and MSPTs. Let $N$ be a
set of terminals embedded in the Euclidean plane and let $T_S$ be an SMT on $N$.

\begin{conjecture}If $T_S$ is full then there exists an MSPT on $N$ that is a degeneracy of $T_S$.
\end{conjecture}

This conjecture implies that an algorithm based on displacing Steiner points of an SMT has the potential for generating optimal MSPTs for a very
large class of terminal configurations. The next conjecture would allow such an algorithm to run in polynomial time.

\begin{conjecture}Let $G$ be any tree topology on $N$ where Steiner points are of degree at least three. Then finding a
tree $T$ on $N$ with the same topology as $G$ and minimizing $\mathrm{beads}(T)$ can be done in polynomial time.
\end{conjecture}

\textbf{Acknowledgement.} The authors wish to thank Jamie Evans for partaking in many fruitful discussions during the development of this paper.
We would also like to thank the referees for their insightful comments.

\end{document}